\documentclass[11pt, a4paper]{amsart}
\usepackage{amsmath}
\usepackage{amsthm}
\usepackage{amsfonts}
\usepackage{amssymb}
\usepackage{verbatim}
\usepackage{graphicx}
\usepackage{geometry}
\usepackage{a4wide}
\usepackage[latin1]{inputenc}

\DeclareMathOperator{\IKM}{\textit{I}_{KM}}
\DeclareMathOperator{\IM}{\textit{I}_{M}}

\DeclareMathOperator{\reg}{reg}

\DeclareMathOperator{\SL}{SL}
\DeclareMathOperator{\Mp}{Mp}

\DeclareMathOperator{\Z}{\mathbb{Z}}
\DeclareMathOperator{\R}{\mathbb{R}}
\DeclareMathOperator{\C}{\mathbb{C}}
\DeclareMathOperator{\Q}{\mathbb{Q}}
\renewcommand{\H}{\mathbb{H}}
\DeclareMathOperator{\Aut}{Aut}
\DeclareMathOperator{\tr}{tr}

\DeclareMathOperator{\calQ}{\mathcal{Q}}

\DeclareMathOperator{\e}{\mathfrak{e}}

	\newtheorem{Satz}{Satz}[section]
	\newtheorem{Theorem}[Satz]{Theorem}
	\newtheorem{Lemma}[Satz]{Lemma}
	 
	\newtheorem{Corollary}[Satz]{Corollary}
	
	\theoremstyle{definition} 
	
	\newtheorem{Example}[Satz]{Example}
	\newtheorem{Remark}[Satz]{Remark}
\date{\today}
\author{Jan Hendrik Bruinier and Markus Schwagenscheidt}
\title[Algebraic formulas for mock theta functions]{Algebraic formulas for the coefficients of mock theta functions and Weyl vectors of Borcherds products}

\address{Fachbereich Mathematik, Technische Universit\"at Darmstadt, Schlossgartenstra{\ss}e 7, D--64289 Darmstadt, Germany}
\email{bruinier@mathematik.tu-darmstadt.de}
\email{schwagenscheidt@mathematik.tu-darmstadt.de}

\thanks{The authors are partially supported by the DFG Research Unit FOR 1920 \lq Symmetry, Geometry and Arithmetic\rq.
}

\begin{document}

\begin{abstract}
We present some applications of the Kudla-Millson and the Millson theta lift. The two lifts map weakly holomorphic modular functions to vector valued harmonic Maass forms of weight $3/2$ and $1/2$, respectively. We give finite algebraic formulas for the coefficients of Ramanujan's mock theta functions $f(q)$ and $\omega(q)$ in terms of traces of CM-values of a weakly holomorphic modular function. Further, we construct vector valued harmonic Maass forms whose shadows are unary theta functions, and whose holomorphic parts have rational coefficients. This yields a rationality result for the coefficients of mock theta functions, i.e., harmonic Maass forms whose shadows lie in the space of unary theta functions. Moreover, the harmonic Maass forms we construct can be used to evaluate the Petersson inner products of unary theta functions with harmonic Maass forms, giving formulas and rationality results for the Weyl vectors of Borcherds products.
\end{abstract}

\maketitle

	\section{Introduction}
	
	Over the last two decades, starting with the fundamental work of Borcherds \cite{Borcherds}, theta lifts between spaces of integral and half-integral weight weakly holomorphic modular forms have become a powerful tool in number theory. For example, Funke and the first author \cite{BruinierFunke06} used the Kudla-Millson theta lift from weight $0$ to weight $3/2$ harmonic Maass forms to give a new proof and generalizations of Zagier's \cite{ZagierTraces} famous result on the modularity of the generating series of traces of singular moduli. Further, in \cite{BruinierOnoAlgebraicFormulas}, Ono and the first author used a variant of the Kudla-Millson theta lift to find a finite algebraic formula for the partition function $p(n)$ in terms of traces of CM-values of a certain non-holomorphic modular function. Recently, a similar theta lift was used in \cite{AGOR} to prove a refinement of a theorem of \cite{BruinierOnoHeegnerDivisors} connecting the vanishing of the central derivative of the twisted $L$-function of an even weight newform and the rationality of some coefficient of the holomorphic part of a half-integral weight harmonic Maass form. This so-called Millson theta lift, which maps weight $0$ to weight $1/2$ harmonic Maass forms, was studied in great detail by Alfes-Neumann in her thesis \cite{Alfesdiss}, and by Alfes-Neumann and the second author in \cite{AlfesSchwagenscheidt}. The aim of this paper is to present some applications of the Kudla-Millson lift from \cite{BruinierFunke06} and the Millson lift studied in \cite{AlfesSchwagenscheidt}. We discuss the following three applications.
	\subsection*{Algebraic formulas for Ramanujan's mock theta functions} We give finite algebraic formulas for the coefficients of Ramanujan's order $3$ mock theta functions $f(q)$ and $\omega(q)$ in terms of traces of CM-values of a weakly holomorphic modular function (see Theorem~\ref{MockThetaFormulas}). For example, we show that the coefficients $a_{f}(n), n \geq 1$, of Ramanujan's mock theta function
		\begin{align}\label{Ramanujanf}
		f(q) &= 1+\sum_{n=1}^{\infty}\frac{q^{n^{2}}}{(1+q)^{2}(1+q^{2})^{2}\cdots (1+q^{n})^{2}} = 1 + \sum_{n=1}^{\infty}a_{f}(n)q^{n}
		\end{align}
		 are given by
		\begin{align*}
		a_{f}(n) = -\frac{1}{\sqrt{24n-1}}\Im\bigg(\sum_{Q \in \mathcal{Q}_{n}}\frac{F(z_{Q})}{\omega_{Q}}\bigg),
		\end{align*}
		where 
		\begin{align}\label{InputF}
		F(z) = -\frac{1}{40}\cdot\frac{E_{4}(z) + 4E_{4}(2z) - 9E_{4}(3z) - 36E_{4}(6z)}{(\eta(z)\eta(2z)\eta(3z)\eta(6z))^{2}} = q^{-1}-4-83q-296q^{2}+ \dots
		\end{align}
		is a $\Gamma_{0}(6)$-invariant weakly holomorphic modular function, $\mathcal{Q}_{n}$ is the (finite) set of $\Gamma_{0}(6)$-equivalence classes of positive definite integral binary quadratic forms $Q(x,y) = ax^{2}+ bxy+cy^{2}$ of discriminant $1-24n$ with $6 \mid a$ and $b \equiv 1(12)$, $z_{Q} \in \H$ is the CM-point characterized by $Q(z_{Q},1) = 0$, and $\omega_{Q}$ is half the order of the stabilizer of $Q$ in $\Gamma_{0}(6)$. Moreover, $E_{4}$ denotes the normalized Eisenstein series of weight $4$ for $\SL_{2}(\Z)$ and $\eta = q^{1/24}\prod_{n = 1}^{\infty}(1-q^{n})$ is the Dedekind eta function. 
		
		For the proof, we use Zwegers' \cite{Zwegers} realization of Ramanujan's mock theta functions as the holomorphic parts of vector valued harmonic Maass forms of weight $1/2$, then construct the corresponding harmonic Maass form as the Millson lift of $F$, and finally obtain the formula by comparing Fourier coefficients.
		\subsection*{Rationality results for harmonic Maass forms} By applying the Kudla-Millson and the Millson theta lifts to a suitable weakly holomorphic input function, we construct harmonic Maass forms of weight $3/2$ and $1/2$ whose images under the differential operator $\xi_{k} = 2iv^{k}\overline{\frac{\partial}{\partial \bar{\tau}}}$ are vector valued unary theta functions of weight $1/2$ and $3/2$, and whose holomorphic parts (which are mock modular forms) are given by traces of CM-values of the input function (see Theorem~\ref{ThetaLifts}). This implies that these mock modular forms have rational coefficients (see Theorem \ref{RationalHolomorphicPart}), which in turn yields a rationality result for the holomorphic parts of harmonic Maass forms that map to the space of unary theta functions under $\xi$ (see Theorem~\ref{AlgebraicCoefficients}). 
		
		More specifically, we show that if $f$ is a vector valued harmonic Maass form of weight $1/2$ whose principal part is defined over a number field $K$, and whose shadow lies in the space of unary theta functions, then all coefficients of the holomorphic part of $f$ lie in $K$. This contrasts a conjecture of Ono and the first author \cite{BruinierOnoHeegnerDivisors}, stating that if $f$ is a harmonic Maass form of weight $1/2$ whose shadow is orthogonal to the space of unary theta functions, then all but a set of density $0$ of the non-vanishing coefficients of the holomorphic part of $f$ should be transcendental. 
		\subsection*{Inner product formulas and Weyl vectors of Borcherds products} We use our $\xi$-preimages to evaluate the regularized Petersson inner product of a harmonic Maass form $f$ and a unary theta function of weight $1/2$ (see Theorem~\ref{InnerProductFormulas}), and apply this to compute the Weyl vectors of the Borcherds lift of $f$ (see Corollary~\ref{WeylVectors}). For example, for $N = 1$ the Borcherds product associated with a weakly holomorphic modular form $f = \sum_{n \gg -\infty}c_{f}(n)q^{n}$ of weight $1/2$ for $\Gamma_{0}(4)$ in the Kohnen plus space with rational coefficients and integral principal part is given by
		\[
		\Psi(z,f) = q^{\rho_{f}}\prod_{n=1}^{\infty}(1-q^{n})^{c_{f}(n^{2})}, \qquad (q = e^{2\pi i z}),
		\] 
		where $\rho_{f}$ is the so-called Weyl vector of $f$. The product $\Psi(z,f)$ is a modular form of weight $c_{f}(0)$ for $\SL_{2}(\Z)$, whose divisor on $\H$ is a Heegner divisor. For higher level $N$, the orders of $\Psi(z,f)$ at the cusps of $\Gamma_{0}(N)$ are determined by Weyl vectors associated to the cusps. These vectors are essentially given by regularized inner products of $f$ with a unary theta function of weight $1/2$, and can be explicitly evaluated in terms of the coefficients of the holomorphic part of $f$ and the coefficients of the holomorphic part of a $\xi$-primage of the unary theta function. In particular, we show that all Weyl vectors associated to a harmonic Maass form with rational holomorphic coefficients are rational (see Corollary \ref{WeylVectorsRationality}).
	\\~
	
	We start with a section on the necessary definitions and notations, and then discuss each of the three applications in a separate section.

		The starting point of this work was a talk given by Yingkun Li in march 2016 about his construction of real-analytic theta series of weight $1/2$. We thank him cordially for his inspiration and for helpful discussions.

	\section{Preliminaries}

		We first need to set up some background material about unary theta functions, harmonic Maass forms, and traces of CM-values of modular functions. Throughout this work, we let $N$ be a positive integer. 
		
		\subsection{Unary theta function for the Weil representation} We consider the positive definite lattice $\Z$ with the quadratic form $n\mapsto Nn^{2}$. Its dual lattice is $\frac{1}{2N}\Z$, so its discriminant group can be identified with $\Z/2N\Z$, equipped with the finite quadratic form ${r \mapsto r^{2}/4N \text{ mod } \Z}$. For $r \in \Z/2N\Z$ we let $\e_{r}$ be the standard basis vectors of the group ring $\C[\Z/2N\Z]$, and we let $\langle \cdot,\cdot \rangle$ be the standard inner product on $\C[\Z/2N\Z]$ which satisfies $\langle \e_{r},\e_{r'}\rangle = \delta_{r,r'}$ and is antilinear in the second variable. The associated Weil representation is defined on the generators $T= \left(\left(\begin{smallmatrix}1 & 1 \\0  & 1 \end{smallmatrix}\right),1\right)$ and $S = \left(\left(\begin{smallmatrix}0 & -1 \\1  & 0 \end{smallmatrix}\right),\sqrt{\tau}\right)$ of the metaplectic group $\Mp_{2}(\Z)$ by
		\begin{align*}
		\rho(T)\e_{r} &= e(Q(r))\e_{r}, \qquad  \rho(S)\e_{r} = \frac{e(-1/8)}{\sqrt{2N}}\sum_{r' (2N)}e(-(r,r'))\e_{r'},
		\end{align*}
		where $e(z) = e^{2\pi i z}$ for $z\in \C$. We consider the unary theta functions
		\begin{align}\label{UnaryTheta12}
		\theta_{1/2}(\tau) = \sum_{r(2N)}\sum_{\substack{b \in \Z \\ b \equiv r(2N)}}q^{b^{2}/4N}\e_{r} \qquad \text{and} \qquad 		\theta_{3/2}(\tau) = \sum_{r(2N)}\sum_{\substack{b \in \Z \\ b \equiv r(2N)}}bq^{b^{2}/4N}\e_{r},
		\end{align}
		where $q = e^{2\pi i \tau}$. They are holomorphic vector valued modular forms of weight $1/2$ and $3/2$ for $\rho$, as can be seen using \cite{Borcherds}, Theorem~4.1, for example. From a more conceptual point of view, these are the theta functions associated to the archimedian Schwartz functions $e^{-\pi x^{2}}$ and $x e^{-\pi x^{2}}$ in the Schr\"odinger model of the Weil representation associated with a one-dimensional quadratic space, compare \cite{BruinierFunke04}, Section 2. We sometimes write $\rho_{N}, \theta_{1/2,N}$ and $\theta_{3/2,N}$ if we want to emphasize the dependence on $N$.
		
		Note that for $k \in 1/2+\Z$ the space $M_{k,\rho}$ of holomorphic modular forms of weight $k$ for $\rho$ is isomorphic to the space $J_{k+1/2,N}^{*}$ of skew-holomorphic Jacobi forms of weight $k+1/2$ and index $N$, and $M_{k,\bar{\rho}}$ is isomorphic to the space $J_{k+1/2,N}$ of holomorphic Jacobi forms of weight $k+1/2$ and index $N$ (see \cite{EichlerZagier}, Section~5). In particular, one can view $\theta_{1/2}$ and $\theta_{3/2}$ as skew-holomorphic Jacobi forms of weight $1$ and $2$. Further, we will make frequent use of a result by Skoruppa which asserts that $M_{1/2,\bar{\rho}} \cong J_{1,N} = \{0\}$ for all $N$ (see \cite{EichlerZagier}, Theorem~5.7). 
		
		Besides the unary theta functions $\theta_{1/2}$ and $\theta_{3/2}$ themselves, we also want to regard modular forms which arise from these series by applications of certain simple operators as theta functions. There are two classes of such operators, which we describe now.
		
		The automorphism group $\Aut(\Z/2N\Z)$ acts on vector valued modular forms $f = \sum_{r}f_{r}\e_{r}$ for $\rho$ or $\bar{\rho}$ by $f^{\sigma} = \sum_{r}f_{r}\e_{\sigma(r)}$. The elements of $\Aut(\Z/2N\Z)$ are all involutions, also called \emph{Atkin-Lehner involutions}, and correspond to the exact divisors $c\mid \mid N$ (i.e., $c\mid N$ and $(c,N/c) =1$). The automorphism $\sigma_{c}$ corresponding to $c$ is defined by the equations
	\begin{align}\label{AtkinLehnerInvolution}
	\sigma_{c}(r) \equiv -r \ (2c) \quad \text{and} \quad \sigma_{c}(r) \equiv r \ (2N/c)
	\end{align}
	for $r \in \Z/2N\Z$, compare \cite{EichlerZagier}, Theorem~5.2. Note that the Atkin-Lehner involutions only permute the components of a vector valued modular form.
	
	For each positive integer $d$ there is a natural operator which maps modular forms for $\rho_{N}$ to forms for $\rho_{Nd^{2}}$ (see \cite{ScheithauerModularforms}, Section 4). Under the isomorphism $J_{k+1/2,N}^{*}\cong M_{k,\rho_{N}}$ it corresponds to the index raising operator $U_{d}$ defined in \cite{EichlerZagier}, Section 4. Its action on the Fourier expansion of a holomorphic modular form $f = \sum_{r(2N) }\sum_{\substack{D \equiv r^{2}(4N)} }c_{f}(D,r)e\left( \frac{D}{4N}\tau\right)\e_{r}$ for $\rho_{N}$ is given by
	\begin{align}
	f|U_{d} &= \sum_{\substack{r(2Nd^{2}) \\ r \equiv 0(d)}}\sum_{\substack{D \geq 0 \\ D \equiv r^{2} (4Nd^{2})} }c_{f}(D/d^{2},r/d)e\left( \frac{D}{4Nd^{2}}\tau\right)\e_{r}. \label{UdOperator}
	\end{align}
	In particular, $U_{d}$ only distributes the components of $f$ in a certain way, but does not change the set of Fourier coefficients of $f$.
	
	 With our application to Weyl vectors of Borcherds products already in mind (compare \eqref{WeylVectorAtCusp}), we now define the \emph{space of unary theta functions of weight $1/2$ for} $\rho_{N}$ as
	\[
	\sum_{\substack{d^{2} \mid N}}\sum_{c \mid \mid \frac{N}{d^{2}}}\C\theta_{1/2,N/d^{2}}^{\sigma_{c}}|U_{d}\, .
	\]
	The space of unary theta functions of weight $3/2$ for $\rho_{N}$ is defined analogously. We now show that the space of unary theta functions of weight $1/2$ agrees with the whole space $M_{1/2,\rho}$.
	
	\begin{Lemma}
		Let $\mathcal{D}(N)$ be the set of all positive divisors of $N$ modulo the equivalence relation $c \sim N/c$. Then the theta functions
		\[
		\theta_{1/2,N/(c,N/c)^{2}}^{\sigma_{c/(c,N/c)}}|U_{(c,N/c)}, \quad c \in \mathcal{D}(N),
		\]
		form a basis of $M_{1/2,\rho}$.
	\end{Lemma}
		
	\begin{proof}
		Using a dimension formula for $J_{1,N}^{*} \cong M_{1/2,\rho}$ (see \cite{SkoruppaZagier}, p. 130), we see that the number of given theta functions agrees with the dimension of $M_{1/2,\rho}$. On the other hand, by looking at the constant terms and the coefficients at $q^{d^{2}/4N}\e_{d}$ for $d \in \mathcal{D}(N)$, it is easy to see that the given functions are linearly independent.
	\end{proof}
	
	\begin{Remark}\label{SerreStarkRemark} This result resembles the Serre-Stark theorem (see \cite{SerreStark}) which states that every scalar valued modular form of weight $1/2$ is a linear combination of elementary theta functions $\theta_{\psi,m}(\tau) = \sum_{n \in \Z}\psi(n)n^{\nu}q^{mn^{2}}$, where $\nu = 0$ and $\psi$ is a primitive Dirichlet character with $\psi(-1) = 1$ (for $\nu = 1$ and $\psi(-1)= -1$ the theta function $\theta_{\psi,m}$ is a cusp form of weight $3/2$). If we split $\theta_{\psi,m}(\tau) = \sum_{r (2d)}\psi(r)\sum_{n \equiv r(2d)}n^{\nu}q^{mn^{2}}$, where $d$ is the conductor of $\psi$, we see that the inner sums are just the components of a vector valued unary theta function. Hence, many of the results of the present work, which are formulated for vector valued unary theta functions, immediatly imply the corresponding statements for the classical scalar valued theta functions.
	\end{Remark}
 
 \subsection{Harmonic Maass forms} Recall from \cite{BruinierFunke04} that a smooth function $f: \H \to \C[\Z/2N\Z]$ is called a \emph{harmonic Maass form} of weight $k \in 1/2+\Z$ with respect to $\bar{\rho}$ if it is annihilated by the weight $k$ hyperbolic Laplace operator $\Delta_k$, transforms like a modular form of weight $k$ for $\bar{\rho}$, and grows at most linearly exponentially at $\infty$. We denote the space of such functions by $H_{k,\bar{\rho}}$. Further, we let $M^{!}_{k,\bar{\rho}}$ be the subspace of weakly holomorphic modular forms, consisting of the forms in $H_{k,\bar{\rho}}$ which are holomorphic on $\H$. An important tool in the theory of harmonic Maass forms is the antilinear differential operator $\xi_k f=2iv^{k}\overline{\frac{\partial}{\partial \bar{\tau}}f(\tau)}$. It defines a surjective map $\xi_k: H_{k,\bar{\rho}}\rightarrow M^{!}_{2-k,\rho}$. 
 
In this paper we only consider harmonic Maass forms that map to the space $M_{2-k,\rho}$ of holomorphic modular forms under $\xi_{k}$. Such a form for $\bar{\rho}$ uniquely decomposes into a \emph{holomorphic} and a \emph{non-holomorphic part} $f=f^{+}+f^{-}$ with Fourier expansions
\begin{align*}
f^{+}(\tau)&=\sum\limits_{r(2N)}\sum\limits_{\substack{D\gg \infty \\ D \equiv -r^{2}(4N)}}c_{f}^{+}(D,r)q^{D/4N} \mathfrak{e}_r,
\\
f^{-}(\tau)&=\sum\limits_{r(2N)}\bigg(c_{f}^{-}(0,r)v^{1-k} + \sum\limits_{\substack{D < 0 \\ D \equiv -r^{2}(4N)}}c_{f}^{-}(D,r)\Gamma(1-k,\pi |D|v/N)q^{D/4N}\bigg)\mathfrak{e}_r,
\end{align*}
where $\Gamma(s,x) = \int_{x}^{\infty}t^{s-1}e^{-t}dt$ denotes the incomplete gamma function. The finite Fourier polynomial $\sum_{r(2N)}\sum_{\substack{D \leq 0}}c_{f}^{+}(D,r)q^{D/4N} \mathfrak{e}_r$ is called the \emph{principal part} of $f$. We briefly call the coefficients of the holomorphic part the \emph{holomorphic coefficients} of $f$.
 The holomorphic part $f^{+}$ is sometimes called a \emph{mock modular form}, and $\xi_{k}f = \xi_{k}f^{-}$ its \emph{shadow}.
 
 We remark that the operators $U_{d}$ and $\sigma_{c}$ defined in the last section also act on harmonic Maass forms, the action on the Fourier expansion being the same, and that they commute with the $\xi$-operator.

The regularized inner product of $f \in M^{!}_{k,\rho}$ and $g \in S_{k,\rho}$ for $k \in 	1/2+\Z$ is defined by
	\begin{align}\label{PeterssonInnerProduct}
	(f,g)^{\reg} = \lim_{T \to \infty} \int_{\mathcal{F}_{T}}\langle f(\tau),g(\tau) \rangle v^{k}\frac{du\, dv}{v^{2}}, \qquad (\tau = u+iv),
	\end{align}
	where $\mathcal{F}_{T} = \{\tau \in \H: |u|\leq \tfrac{1}{2}, |\tau| \geq 1 ,v \leq T\}$ is a truncated fundamental domain for the action of $\SL_{2}(\Z)$ on $\H$. For $k=1/2$ the regularized inner product also converges for $g \in M_{1/2,\rho}$. An application of Stokes' theorem (compare Proposition~3.5 in \cite{BruinierFunke06}) shows that for $f  \in M^{!}_{k,\rho}$ with coefficients $c_{f}(D,r)$, and $G \in H_{2-k,\bar{\rho}}$ with holomorphic coefficients $c_{G}^{+}(D,r)$ and $\xi_{2-k}G=g \in S_{k,\rho}$ (or $g \in M_{1/2,\rho}$ if $k =1/2$), the regularized inner product can be evaluated as
	\begin{align}\label{InnerProductEvaluation}
	(f,g)^{\reg} = (f,\xi_{2-k}G)^{\reg} =\sum_{r(2N)}\sum_{\substack{D \in \Z \\ D \equiv r^{2}(4N)}}c_{f}(D,r)c_{G}^{+}(-D,r).
	\end{align}
	This can be employed to show the following useful lemma.
	
	\begin{Lemma}\label{CuspFormLemma}
		Let $G$ be a harmonic Maass form of weight $2-k \in 1/2+\Z$ for $\rho$ or $\bar{\rho}$ whose principal part vanishes and which maps to a cusp form under $\xi_{2-k}$ (or a holomorphic modular form if $k=1/2$). Then $G$ is a cusp form.
	\end{Lemma}
	
	\begin{proof}
		Suppose that $g = \xi_{2-k} G$ is a cusp form. By \eqref{InnerProductEvaluation} we see that $(g,g)^{\reg} = (g,\xi_{k} G)^{\reg} = 0$, so $g = 0$. This means that $G$ is holomorphic, hence a cusp form. For $k=1/2$ the inner product $(g,g)^{\reg}$ also converges if $g \in M_{1/2,\rho}$, so the same argument applies in this case.
	\end{proof}
	
%	\subsection{Modular functions of level $N$} We let $M_{0}^{!}(N)$ denote the space of modular functions of level $N$, i.e. the space of holomorphic $\Gamma_{0}(N)$-invariant functions $f:\H \to \C$ which may grow linearly exponentially at the cusps. Recall that for each cusp $\ell \in \Q\cup\{\infty\}$ there is a matrix $\sigma_{\ell} \in \SL_{2}(\Z)$ such that $\sigma_{\ell}\infty = \ell$, and that a modular form $F \in M_{0}^{!}(N)$ has Fourier expansions of the shape
%	\[
%	F(\sigma_{\ell}z) = \sum_{n \gg -\infty}a_{\ell}(n)q^{n/N},\qquad (q = e^{2\pi i z}),
%	\]
%	at the cusps. The finite sum $\sum_{n \leq 0}a_{\ell}(n)q^{n/N}$ is called the principal part of $F$ at $\ell$, and $a_{\ell}(0)$ is the constant term of $F$ at $\ell$.

%		
%		Further, for each exact divisor $d \mid \mid N$ there is an Atkin-Lehner involution $W_{d}^{N} = \left(\begin{smallmatrix}d\alpha & \beta \\ N\gamma & d \delta\end{smallmatrix}\right)$, where $\alpha, \beta, \gamma, \delta \in \Z$ are chosen such that $W_{d}^{N}$ has determinant $d$, which acts by $F|W_{d}^{N}:= F(W_{d}^{N}z)$ on $M^{!}_{0}(N)$. For two exact divisors $d,d'$ of $N$ we have $F|W_{d}^{N}|W_{d'}^{N}= F|W_{d * d'}^{N}$, where $d * d' = \frac{dd'}{(d,d')^{2}}$ is again an exact divisor of $N$.

	\subsection{Traces of CM-values and theta lifts} For $k \in \Z$ we let $M_{k}^{!}(N)$ denote the space of weakly holomorphic modular forms of weight $k$ and level $N$, i.e., the space of meromorphic modular forms of weight $k$ for $\Gamma_{0}(N)$ whose poles are supported at the cusps. We denote the subspace of holomorphic forms by $M_{k}(N)$.
	
	For a discriminant $D < 0$ and $r \in \Z$ with $D \equiv r^{2}(4N)$ we let $\calQ_{N,D,r}$ be the set of integral binary quadratic forms $ax^{2}+bxy + cy^{2}$ of discriminant $D = b^{2}-4ac$ with $N \mid a$ and $b \equiv r (2N)$. It splits into the sets $\calQ_{N,D,r}^{+}$ and $\calQ_{N,D,r}^{-}$ of positive and negative definite forms. The group $\Gamma_{0}(N)$ acts on both sets with finitely many orbits, and the order $w_{Q}= \frac{1}{2}|\Gamma_{0}(N)_{Q}|$ is finite. For each $Q \in \calQ_{N,D,r}$ there is an associated Heegner (or CM) point $z_{Q} = (-b+i\sqrt{|D|})/2a$, which is characterized by $Q(z_{Q},1) = 0$. We define two trace functions of $F \in M_{0}^{!}(N)$ by
		\begin{align}\label{traces}
		\tr_{F}^{+}(D,r) = \sum_{Q \in \calQ_{N,D,r}^{+}/\Gamma_{0}(N)}\frac{F(z_{Q})}{w_{Q}} \qquad \text{and} \qquad \tr_{F}^{-}(D,r) = \sum_{Q \in \calQ_{N,D,r}^{-}/\Gamma_{0}(N)}\frac{F(z_{Q})}{w_{Q}}.
		\end{align}
				
		The Kudla-Millson theta lift $\IKM(F,\tau)$ and the Millson theta lift $\IM(F,\tau)$ of a weakly holomorphic modular function $F \in M_{0}^{!}(N)$ are both defined as integrals of the shape
		\[
		I(F,\tau) = \int_{\Gamma_{0}(N)\setminus \H}F(z)\Theta(\tau,z)\frac{dx \, dy}{y^{2}}, \qquad (z = x + iy),
		\]
		where $\Theta(\tau,z)$ is a suitable theta function which is $\Gamma_{0}(N)$-invariant in $z$ and transforms like a modular form of weight $3/2$ (for the Kudla-Millson lift) or weight $1/2$ (for the the Millson lift) for the representation $\bar{\rho}$ in $\tau$. We remark that we renormalize the Millson theta lift from \cite{AlfesSchwagenscheidt} by multiplication with $i/\sqrt{N}$ to simplify the formulas. The most important properties of the two theta lifts are studied in \cite{BruinierFunke06} and \cite{AlfesSchwagenscheidt}. It turns out that both lifts map modular functions for $\Gamma_{0}(N)$ to harmonic Maass forms for $\bar{\rho}$ whose $\xi$-images are linear combinations of unary theta functions. Further, the Fourier expansions of the holomorphic parts of the lifts are given by traces of $F$. We state simplified Fourier expansions of the lifts for special inputs $F$ in Theorem~\ref{ThetaLifts}.

		\section{Algebraic formulas for the coefficients of Ramanujan's $f(q)$ and $\omega(q)$}
		
	As a first application of the Millson theta lift, we find finite algebraic formulas for the coefficients of Ramanujan's third order mock theta functions $f(q)$ defined in \eqref{Ramanujanf} and
%	\begin{align*}
%	f(q) &= 1+\sum_{n=1}^{\infty}\frac{q^{n^{2}}}{(1+q)^{2}(1+q^{2})^{2}\cdots (1+q^{n})^{2}} \\
%	&= 1 + q -2 q^{2} +3q^{3} -3q^{4} -5q^{5} +7q^{6} -6q^{7}+ \dots
%	\end{align*}
%	and
	\begin{align*}
	\omega(q) &=  1+ \sum_{n=1}^{\infty}\frac{q^{2n^{2}+2n}}{(1-q)^{2}(1-q^{3})^{2}\cdots (1-q^{2n+1})^{2}} \\
	&= 1 + 2q + 3q^{2} + 4q^{3} + 6q^{4} + 8q^{5}+ 10q^{6}+14q^{7}+\dots
	\end{align*}
	in terms of the traces of a single modular function. We obtain the following result.

	\begin{Theorem}\label{MockThetaFormulas}
		Let $F \in M_{0}^{!}(6)$ be the function defined in \eqref{InputF}.
		\begin{enumerate}
			\item For $n \geq 1$ the coefficients $a_{f}(n)$ of $f(q)$ are given by
			\begin{align*}
	a_{f}(n) = \frac{i}{2\sqrt{24n-1}}\big(\tr_{F}^{+}(1-24n,1)-\tr_{F}^{-}(1-24n,1)\big).
	\end{align*}
			\item For $n \geq 1$ the coefficients $a_{\omega}(n)$ of $\omega(q)$ are given by
	\begin{align*}
	a_{\omega}(n) =
	\begin{cases}
	 \frac{-i}{8\sqrt{24\frac{n+1}{2}-4}}\left(\tr_{F}^{+}\left(4-24\frac{n+1}{2},2\right) - \tr_{F}^{-}\left(4-24\frac{n+1}{2},2\right)\right), & n \text{ odd}, \\
	\frac{-i}{8\sqrt{24(\frac{n}{2}+1)-16}}\big(\tr_{F}^{+}(16-24\left(\frac{n}{2}+1\right),4) - \tr_{F}^{-}(16-24\left(\frac{n}{2}+1\right),4)\big), & n \text{ even}.
	\end{cases}
	\end{align*}
		\end{enumerate}
	\end{Theorem}

	\begin{Remark}
		\begin{enumerate}
			\item The theorem extends results of Alfes-Neumann (see the example after Theorem~1.3 in \cite{alfes}), who gave similar formulas for the coefficients $a_{f}(n)$ with $1-24n$ being a fundamental discriminant, by looking at the Kudla-Millson theta lift of $F$ and employing a duality result between the Millson and the Kudla-Millson lift.
			\item Using the Kudla-Millson theta lift, Ahlgren and Andersen \cite{AhlgrenAndersen} gave a formula for the smallest parts function in terms of traces of a modular function. The coefficients of Ramanujan's mock theta functions are related to partitions as well. For example, $a_{f}(n)$ is the number of partitions of $n$ with even rank minus the number with odd rank, where the rank of a partition is its largest part minus the number of parts.
			\item On of the main ingredients in the proof is Zwegers' \cite{ZwegersPaper} realization of the mock theta functions $f(q)$ and $\omega(q)$ as the holomorphic parts of the components of a vector valued harmonic Maass form. Thus the same idea works for other mock theta functions as well, for example for the order $5$ and order $7$ mock theta functions treated in Zwegers' thesis \cite{Zwegers}. The details will be part of an ongoing master's thesis.
			\item We checked the above formulas numerically using Sage \cite{Sage}.
		\end{enumerate}
	\end{Remark}
	
	\begin{Example}
		We illustrate our formulas by computing $a_{\omega}(1) = 2$. Note that $\tr_{F}^{-}(D,r) = \overline{\tr_{F}^{+}(D,r)}$ since $F$ has real coefficients, so we obtain
		\[
		a_{\omega}(1) = \frac{-i}{8\sqrt{20}}\left(\tr_{F}^{+}\left(-20,2\right) - \tr_{F}^{-}\left(-20,2\right)\right) = \frac{2}{8\sqrt{20}}\Im\left(\tr_{F}^{+}\left(-20,2\right) \right).
		\]
		A system of representatives of the $\Gamma_{0}(6)$-classes of forms $ax^{2} + bxy + cy^{2}$ with $6 \mid a , a > 0, b \equiv 2 (12)$ and $D = b^{2}-4ac = -20$ is given by the two forms
		\begin{align*}
		Q_{1} = 6x^{2} -10xy +  5y^{2}, \qquad Q_{2} = 42x^2 - 34xy + 7y^2,
		\end{align*}
		and the corresponding CM-points are
		\begin{align*}
		z_{Q_{1}} = \frac{10+i\sqrt{20}}{12},\qquad z_{Q_{2}} = \frac{34+i\sqrt{20}}{84}.
		\end{align*}
		Plugging these values into the Fourier expansion of $F$, we find
		\begin{align*}
		F(z_{Q_{1}}) = F(z_{Q_{2}}) \sim  i\cdot 17.888543820000.
		\end{align*}
		Thus we obtain
		\[
		a_{\omega}(1) \sim \frac{2}{8\sqrt{20}}\cdot 2\cdot 17.888543820000 = 2.000000000000.
		\]
	\end{Example}
	
	\begin{proof}[Proof of Theorem \ref{MockThetaFormulas}] Zwegers \cite{ZwegersPaper} showed that the function
	\[
	(q^{-\frac{1}{24}}f(q),2q^{\frac{1}{3}}\omega(q^{\frac{1}{2}}),2q^{\frac{1}{3}}\omega(-q^{\frac{1}{2}}))^{T}
	\] 
	is the holomorphic part of a vector valued harmonic Maass form $H = (h_{0},h_{1},h_{2})^{T}$, transforming as
	\begin{align*}
	H(\tau+1) = \begin{pmatrix}\zeta_{24}^{-1} & 0 & 0 \\ 0 & 0 & \zeta_{3} \\ 0 & \zeta_{3} & 0 \end{pmatrix}H(\tau), \quad H\left(-\frac{1}{\tau}\right) = \sqrt{-i\tau}\begin{pmatrix}0 & 1 & 0 \\ 1 & 0 & 0 \\ 0 & 0 & -1 \end{pmatrix}H(\tau).
	\end{align*}
	Further, $\xi_{1/2}H$ is a vector consisting of cuspidal unary theta functions of weight $3/2$. 
	
	One can check that 
	\[
	\tilde{H} = (0,h_{0},h_{2}-h_{1},0,-h_{1}-h_{2},-h_{0},0,h_{0},h_{1}+h_{2},0,h_{1}-h_{2},-h_{0})^{T}
	\]
	is a vector valued harmonic Maass form of weight $1/2$ for the dual Weil representation $\bar{\rho}$. We see that its principal part is given by $q^{-1/24}(\e_{1}-\e_{5}+\e_{7}-\e_{11})$.
	
	The function $(\eta(z)\eta(2z)\eta(3z)\eta(6z))^{2}$ in the denominator of $F$ is a cusp form of weight $4$ for $\Gamma_{0}(6)$ which is invariant under all Atkin-Lehner involutions $W_{d}^{6}$ for $d \mid 6$, and the numerator of $F$ equals $E_{4}|(W_{1}^{6}+W_{2}^{6}-W_{3}^{6}-W_{6}^{6})$. Thus $F$ is an eigenfunction of all Atkin-Lehner involutions, with eigenvalue $1$ for $W_{1}^{6}$ and $W_{2}^{6}$, and eigenvalue $-1$ for $W_{3}^{6}$ and $W_{6}^{6}$. In particular, the Fourier expansions of $F$ at the cusps of $\Gamma_{0}(6)$ are essentially the same, up to a possible minus sign. Using the formula for the Fourier expansion of the Millson lift given in Theorem~5.1 of \cite{AlfesSchwagenscheidt}, it is now straightforward to check that $F$ lifts to a harmonic Maass form $\IM(F,\tau)$ of weight $1/2$ for $\bar{\rho}$, having twice the principal part as $\tilde{H}$ (recall that we multiplied the Millson lift by $i/\sqrt{N}$). In view of Lemma~\ref{CuspFormLemma}, this implies that the difference $\tilde{H} - \frac{1}{2}\IM(F,\tau)$ is a cusp form. But $S_{1/2,\bar{\rho}} \cong J^{\text{cusp}}_{1,6} = \{0\}$, so we find $\tilde{H} = \frac{1}{2}\IM(F,\tau)$. The holomorphic coefficients of $\IM(F,\tau)$ at $q^{(24n-r^{2})/24}\e_{r}$ for $r^{2}-24n < 0$ are given by
	\[
	\frac{i}{\sqrt{24n-r^{2}}}\big(\tr_{F}^{+}(r^{2}-24n,r)-\tr_{F}^{-}(r^{2}-24n,r)\big).
	\]
	Comparing the holomorphic parts of $\tilde{H}$ and $\frac{1}{2}\IM(F,\tau)$, we obtain the stated formulas for the coefficients $a_{f}(n)$ and $a_{\omega}(n)$.	
	\end{proof}

	\section{$\xi$-preimages of unary theta functions and rationality results}

In \cite{BringmannFolsomOno} and \cite{BringmannOno}, Bringmann, Folsom and Ono constructed scalar valued harmonic Maass forms of weight $3/2$ and $1/2$ whose shadows are the components of the unary theta functions $\theta_{1/2}$ and $\theta_{3/2}$ above. In both cases, the proof of the modularity of their $\xi$-preimages relies on transformation properties of various hypergeometric functions and $q$-series. Here we construct $\xi$-preimages for both $\theta_{1/2}$ and $\theta_{3/2}$ using the Kudla-Millson and the Millson theta lift of a single weakly holomorphic modular function $F$ for $\Gamma_{0}(N)$. A nice feature of this approach is that the modularity is clear from the construction. Further, the coefficients of the holomorphic parts of these harmonic Maass forms are given by modular traces of $F$, and thus have good arithmetic properties. Therefore, we obtain rationality results for the holomorphic parts of harmonic Maass forms which map to the space of unary theta functions under $\xi$.

Let $\C((q))$ be the ring of formal Laurent series and let $\C[[q]]$ be the ring of formal power series in $q$.
If $f=\sum_n a(n)q^n \in \C((q))$, we call the polynomial 
\[
P_f= \sum_{n\leq 0} a(n)q^{n}\in \C[q^{-1}]
\]
the \emph{principal part} of $f$. There is a bilinear pairing 
\[
\C((q))\times \C[[q]]\to \C, \quad (f,g)\mapsto \{f,g\} := \text{coefficient of $q^0$  
of $f\cdot g$}.  
\]
It only depends on the principal part of $f$. 

Let $k>0$ be an even integer.	
We denote by $M_{k}^{!,\infty}(N) \subset M_{k}^{!}(N)$ the subspace of those weakly holomorphic modular forms which vanish at all cusps different from $\infty$.
We view the space $M_{2-k}^{!,\infty}(N)$ as a subspace of $\C((q))$ and view $M_k(N)$ as a subspace of $\C[[q]]$ by taking $q$-expansions at the cusp $\infty$. 

\begin{Lemma}
\label{lem:var-bor}
Let $P\in \C[q^{-1}]$. There exists an $F\in M_{2-k}^{!,\infty}(N)$ with prescribed principal part $P_F=P$ at the cusp $\infty$, if and only if 
$\{P,g\} = 0$ for all $g\in M_k(N)$.
\end{Lemma}

This can be proved by varying the argument of \cite{borcherdsDuke}, 
Theorem~3.1. By Serre duality it can be shown that the subspace 	$M_{2-k}^{!,\infty}(N)\subset \C((q))$ is the orthogonal complement of $M_k(N)\subset \C[[q]]$ with respect to the pairing $\{\cdot,\cdot\}$.

We use this lemma to construct a suitable input $F$ for the two theta lifts. 
		
	\begin{Lemma}
	\label{ConstructionModularFunction}
		Let $k\in \Z_{>0}$ be even. There exists an $F(z) = \sum_{n \gg -\infty}a(n)q^{n}\in M_{2-k}^{!,\infty}(N)$ with the following properties:
		\begin{enumerate}
			\item The Fourier coefficients $a(n)$ of $F$ at $\infty$ lie in $\Q$.
			\item The constant term $a(0)$ of $F$ at $\infty$ is non-zero.
		\end{enumerate}
\end{Lemma}

\begin{proof}
If $N=1$ and $k=2$ then $M_k(N)$ is trivial. In this case we can take $F=1$. Therefore we exclude this case from now on, so that $M_k(N)$ is non-trivial. We let $M_{k,0}(N)\subset M_{k}(N)$ be the codimension $1$ subspace of those modular forms which vanish at the cusp $\infty$.

Since the cusp at $\infty$ of $X_0(N)$ is defined over $\Q$,
there exists
an $E=\sum_{n\geq 0} c(n)q^n\in M_k(N)$ with rational coefficients which has value $1$  at $\infty$, i.e., $c(0)=1$. (Such an $E$ can be obtained explicitly as a linear combination of the Eisenstein series at the cusps $\infty$ and $0$.)
It is well known that $M_k(N)$ has a basis consisting of modular forms with rational coefficients. Using $E$, we see that the space $M_{k,0}(N)$ also has a basis $g_1,\dots ,g_d$ consisting of forms with rational coefficients. Moreover, we have 
\[
M_k(N)=M_{k,0}(N)\oplus \C E.
\]

The linear map $M_k(N)\to \C[[q]]/\C[q]$ induced by mapping a modular form to its $q$-expansion is injective. Hence, the images of $g_1,\dots,g_d$ and $E$ are linearly independent. Consequently, there exists a polynomial
\[
P_0=\sum_{n<0} a(n)q^{n}\in q^{-1}\Q[q^{-1}]
\]
such that 
\begin{align*}
\{ P_0, g_i\} &= 0,\quad \text{for $i=1,\dots,d$,} \qquad \text{and} \qquad \{ P_0, E\} = -1.
\end{align*}
Put $P=P_0+1\in \Q[q^{-1}]$. Then we have 
\begin{align*}
\{ P, g_i\} &= 0,\quad \text{for $i=1,\dots,d$,} \qquad \text{and} \qquad \{ P, E\} = 0,
\end{align*}
and therefore $\{P,g\}=0$ for all $g\in M_k(N)$.

According to Lemma~\ref{lem:var-bor} there exists an $F\in M_{2-k}^{!,\infty}(N)$ with principal part $P_F=P$. We denote the Fourier expansion of $F$ by 
$F=\sum_n a(n) q^n$. The group $\Aut(\C/\Q)$ acts on $M_{2-k}^!(N)$ by conjugation of the Fourier coefficients. Under the action of $\Aut(\C/\Q)$ on $X_0(N)$ the cusp at $\infty$ is fixed, and the other cusps are permuted among themselves. Hence $\Aut(\C/\Q)$ also acts on $M_{2-k}^{!,\infty}(N)$ by conjugation of the Fourier coefficients. Consequently, for $\sigma\in \Aut(\C/\Q)$, the form $F^\sigma$ also belongs to $M_{2-k}^{!,\infty}(N)$. Since $P_F\in \Q[q^{-1}]$, the form $F-F^\sigma$ has vanishing principal part and therefore vanishes identically. We find that $a(n)=a(n)^\sigma$ for all $n\in \Z$.
Therefore, all Fourier coefficients of $F$ are rational.
\end{proof}

	We now use the modular function $F\in M_{0}^{!,\infty}(N)$ constructed in Lemma~\ref{ConstructionModularFunction} as an input for the Kudla-Millson and the Millson theta lift. The following theorem is just a straightforward simplification of Theorem~4.5 from \cite{BruinierFunke06} and Theorem~5.1 from \cite{AlfesSchwagenscheidt}, and will thus not be proved here. In order to simplify the formulas, we multiply the expansion of the Millson lift given in \cite{AlfesSchwagenscheidt} by $i/\sqrt{N}$.
	
	\begin{Theorem}\label{ThetaLifts}
		Let $F(z) = \sum_{n \gg -\infty}a(n)q^{n}\in M_{0}^{!,\infty}(N)$ be as in Lemma~\ref{ConstructionModularFunction} and $r \in \Z/2N\Z$.
		\begin{enumerate}	
		\item The function
	\begin{align*}
	\IKM(F,\tau)_{r}^{+} &= \sum_{\substack{D < 0 \\ D \equiv r^{2}(4N)}}\big(\tr_{F}^{+}(D,r)+\tr_{F}^{-}(D,r)\big)q^{-D/4N} \\
	&\quad + 4\delta_{0,r}\sum_{n \geq 0}a(-n)\sigma_{1}(n) - \sum_{b > 0}b\big(\delta_{b,r}+\delta_{b,-r} \big)\sum_{n>0}a(-bn)q^{-b^{2}/4N} 
	\end{align*}
	is the $r$-component of the holomorphic part of a harmonic Maass form of weight $3/2$ for $\bar{\rho}$ with
	\[
	\xi_{3/2}\left(\IKM(F,\tau)\right) = -\frac{\sqrt{N}}{4\pi}a(0)\theta_{1/2}(\tau).
	\]
	Here $\delta_{r,r'}$ equals $1$ if $r \equiv r' (2N)$ and $0$ otherwise, and $\sigma_{1}(0) = -\frac{1}{24}$.
	\item The function
	\begin{align*}
	\IM(F,\tau)_{r}^{+} &= \sum_{\substack{D < 0 \\ D \equiv r^{2}(4N)}}\frac{i}{\sqrt{|D|}}\big(\tr_{F}^{+}(D,r)-\tr_{F}^{-}(D,r)\big)q^{-D/4N} \\
	&\quad +\sum_{b > 0}\big(\delta_{b,r}-\delta_{b,-r} \big)\sum_{n>0}a(-bn)q^{-b^{2}/4N} 
	\end{align*}
	is the $r$-component of the holomorphic part of a harmonic Maass form of weight $1/2$ for $\bar{\rho}$ with
	\[
	\xi_{1/2}\left(\IM(F,\tau)\right) = -\frac{1}{2\sqrt{N}}a(0)\theta_{3/2}(\tau).
	\]
	\end{enumerate}
	\end{Theorem}

	\begin{Remark}
	The Kudla-Millson lift of the constant $1$-function gives a generalization of Zagier's non-holomorphic Eisenstein series of weight $3/2$ from \cite{ZagierEisensteinSeries} to arbitrary level $N$, see also \cite{BruinierFunke06}, Remark 4.6. The $\xi$-image of the Eisenstein series is a linear combination of unary theta series associated to lattices $(\Z,n \mapsto dn^{2})$ with $d \mid N$, and is invariant under all Atkin-Lehner involutions. Since $\theta_{1/2}$ is only Atkin-Lehner invariant if $N = 1$ or $N = p$ is prime, we can \emph{not} take the Eisenstein series as a $\xi$-preimage of $\theta_{1/2}$ in general. Also note that usually the principal parts of the harmonic Maass forms given above are non-zero.
	\end{Remark}

	By the theory of complex multiplication, the rationality properties of traces of weakly holomorphic modular functions are well understood. Therefore, we obtain the following result on the rationality of the holomorphic parts of the harmonic Maass forms given above.
	
	\begin{Theorem}\label{RationalHolomorphicPart}
		Let $F \in M_{0}^{!}(N)$ and suppose that the Fourier coefficients of $F$ at $\infty$ lie in $\Z$ and the expansions at all other cusps have coefficients in $\Z[\zeta_{N}]$. Then for $D \equiv r^{2}(4N), D < 0,$ the numbers
			\begin{align}\label{IntegralTraces}
			6 \big(\tr_{F}^{+}(D,r)+\tr_{F}^{-}(D,r)\big) \quad \text{and} \quad 6 t  \frac{i}{\sqrt{|D|}}\big(\tr_{F}^{+}(D,r)-\tr_{F}^{-}(D,r)\big)
			\end{align}
			are rational integers, where $D = t^{2}D_{0}$ with a negative fundamental discriminant $D_{0}$.
	\end{Theorem}
	
	\begin{proof}
		The assumption on the integrality of $F$ at $\infty$ implies that $F \in \Q(j,j_{N})$. By the theory of complex multiplication (see Theorem~4.1 in \cite{BruinierOnoAlgebraicFormulas}), the values $F(z_{Q})$ of $F$ at Heegner points $z_{Q}$ of discriminant $D$ lie in the ring class field of the order $\mathcal{O}_{D}$ over $\Q(\sqrt{D})$. Further, Lemma~4.3 in \cite{BruinierOnoAlgebraicFormulas} asserts that the values $F(z_{Q})$ are algebraic integers. The Galois group of the ring class field of $\mathcal{O}_{D}$ over $\Q(\sqrt{D})$ permutes the Heegner points occuring in $\tr_{F}^{+}(D,r)$ and $\tr_{F}^{-}(D,r)$, see \cite{Gross}. It follows that $6\tr_{F}^{+}(D,r)$ and $6\tr_{F}^{-}(D,r)$ are algebraic integers in $\Q(\sqrt{D})$, where the factor $6$ was added to get rid of possible factors $w_{Q}$ in the denominator. Using that $F$ has rational coefficients at $\infty$, we see that $\overline{\tr_{F}^{-}(D,r)} = \tr_{F}^{+}(D,r)$, and thus
		\begin{align*}
		&\tr_{F}^{+}(D,r)+\tr_{F}^{-}(D,r) \in \Q \quad \text{and} \quad \tr_{F}^{+}(D,r)-\tr_{F}^{-}(D,r) \in \sqrt{D}\Q.
		\end{align*}
		This implies that the quantities in \eqref{IntegralTraces} are rational integers.
	\end{proof}
	
		\begin{Remark} In all the numerical examples we looked at, the numbers in \eqref{IntegralTraces} were already integers without the factors $6$ and $6 t$. Possibly, this is always the case.	
		\end{Remark}
	
		Combining Theorem~\ref{ThetaLifts} and Theorem~\ref{RationalHolomorphicPart} we obtain that if $F \in M_{0}^{!,\infty}(N)$ is as in Lemma~\ref{ConstructionModularFunction} and has rational principal part at $\infty$, then the holomorphic parts of $\IKM(F,\tau)$ and $\IM(F,\tau)$ have rational Fourier coefficients. This rationality result is remarkable since the holomorphic coefficients of a harmonic Maass form of weight $1/2$ which does not map to the space of unary theta functions under $\xi_{1/2}$ are conjectured to be transcendental almost always, see the conjecture and Corollary~1.4 in the introduction of \cite{BruinierOnoHeegnerDivisors}.

	\begin{Theorem}\label{AlgebraicCoefficients} Let $K$ be a number field and let $H_{k,\bar{\rho}}(K)$ be the subspace of $H_{k,\bar{\rho}}$ consisting of forms whose principal part is defined over $K$.
		\begin{enumerate}
			\item Let $f \in H_{1/2,\bar{\rho}}(K)$ and suppose that $f$ is mapped  to the space of unary theta functions by $\xi_{1/2}$. Then the coefficients of the holomorphic part $f^{+}$ of $f$ lie in $K$.
			\item Let $f \in H_{3/2,\bar{\rho}}(K)$ and suppose that $f$ is mapped  to the space of unary theta functions by $\xi_{3/2}$. Then there is a cusp form $f' \in S_{3/2,\bar{\rho}}$ such that the coefficients of $f^{+}-f'$ lie in $K$.
		\end{enumerate}
	\end{Theorem}
	
	\begin{Remark}
		\begin{enumerate}
		\item The corresponding statement for the spaces $H_{1/2,\rho}$ and $H_{3/2,\rho}$ is also true, but a little less interesting. Since there are no unary theta functions for $\bar{\rho}$, it just says that a weakly holomorphic modular form of weight $1/2$ or $3/2$ for $\rho$, whose principal part is defined over $K$, has coefficients in $K$ up to addition of a cusp form. This follows immediatly from the fact that the spaces $M^{!}_{1/2,\rho}$ and $M^{!}_{3/2,\rho}$ have bases with rational coefficients (see \cite{McGraw}).
			\item We have formulated the theorem using vector valued modular forms for the Weil representation since it best suits our applications. In view of Remark \ref{SerreStarkRemark} the theorem implies a rationality result for the $\xi$-preimages of the scalar valued theta series $\theta_{\psi,m}(\tau)$ as well.
			\item The holomorphic projection method developed in \cite{IRR} yields a recursive formula for the coefficients of the holomorphic part of $f$. It should also be possible to prove the above theorem using this technique.
		\end{enumerate}
	\end{Remark}
	
	\begin{proof}[Proof of Theorem \ref{AlgebraicCoefficients}]
		We only prove the first claim, since the second one is similar. Let $f \in H_{1/2,\bar{\rho}}(K)$ and suppose that $\xi_{1/2}f$ lies in the space of unary theta functions. By Theorem~\ref{ThetaLifts}, there is some $h \in H_{1/2,\bar{\rho}}$, which is a linear combination of harmonic Maass forms $h_{j}$ with rational holomorphic parts, such that $\xi_{1/2}f = \xi_{1/2}h$, i.e., $f -h$ is weakly holomorphic. We can write $f - h$ as a linear combination of forms $g_{i}$ with rational coefficients (see \cite{McGraw}). Having $f$ written in terms of the $g_{i}$ and $h_{j}$, we consider the system of linear equations obtained from comparing the principal parts. It is defined over $K$ and has a solution in $\C$, so we can also solve it over $K$. Thereby we obtain a harmonic Maass form $\tilde{f}$ that has the same principal part as $f$ and still maps to the space of unary theta functions under $\xi_{1/2}$, but is now a linear combination of the $g_{i}$ and $h_{j}$ over $K$. In particular, the coefficients of $\tilde{f}^{+}$ lie in $K$. Then $f -\tilde{f}$ is a harmonic Maass forms which has vanishing principal part and maps to the space of unary theta function under $\xi_{1/2}$. By Lemma~\ref{CuspFormLemma}, this implies that $f-\tilde{f}$ is a cusp form. But $S_{1/2,\bar{\rho}} \cong J_{1,N}^{\text{cusp}}= \{0\}$, so $f = \tilde{f}$, and thus $f^{+}$ has coefficients in $K$.
	\end{proof}

	\section{Regularized inner products and Weyl vectors of Borcherds products}

	The harmonic Maass forms constructed in Theorem~\ref{ThetaLifts} can be used to evaluate the regularized Petersson inner product of the unary theta functions $\theta_{1/2}$ and $\theta_{3/2}$ with harmonic Maass forms whose shadows are cusp forms.

	\begin{Theorem}\label{InnerProductFormulas}
	Let $F(z) = \sum_{n \gg -\infty}a(n)q^{n}\in M_{0}^{!,\infty}(N)$ be as in Lemma~\ref{ConstructionModularFunction}.
		\begin{enumerate}
			\item Let $f \in H_{1/2,\rho}$ with holomorphic coefficients $c_{f}^{+}(D,r)$, where $D \equiv r^{2}(4N)$, and suppose that $\xi_{1/2}f \in S_{3/2,\bar{\rho}}$. Then 
			\begin{align*}
			-\frac{\sqrt{N}}{4\pi}a(0)(f,\theta_{1/2})^{\reg} &= \sum_{r(2N)}\sum_{\substack{D < 0\\ D \equiv r^{2}(4N)}}c_{f}^{+}(D,r)\left( \tr_{F}^{+}(D,r) + \tr_{F}^{-}(D,r)\right) \\ 
			& \quad + 4c_{f}^{+}(0,0)\sum_{n\geq 0}a(-n)\sigma_{1}(n) -2\sum_{b > 0}c_{f}^{+}(b^{2},b)b\sum_{n > 0}a(-bn).
			\end{align*}
			\item Let $f \in H_{3/2,\rho}$ with holomorphic coefficients $c_{f}^{+}(D,r)$, where $D \equiv r^{2}(4N)$, and suppose that $\xi_{3/2}f \in S_{1/2,\bar{\rho}}$. Then
			\begin{align*}
			-\frac{1}{2\sqrt{N}}a(0)(f,\theta_{3/2})^{\reg} &= \sum_{r(2N)}\sum_{\substack{D < 0 \\ D \equiv r^{2}(4N)}}c^{+}_{f}(D,r)\frac{i}{\sqrt{|D|}}\left( \tr_{F}^{+}(D,r) - \tr_{F}^{-}(D,r)\right) \\
			& \quad + 2\sum_{b>0}c_{f}^{+}(b^{2},b)\sum_{n>0}a(-bn).
			\end{align*}
		\end{enumerate}
	\end{Theorem}
	
	\begin{proof}
		We show the formula for $\theta_{1/2}$. Using Stokes' theorem, we see as in the proof of \cite{BruinierFunke04}, Proposition~3.5, that
		\begin{align*}
		(f,\theta_{1/2})^{\reg} =  -(\IKM(F,\tau),\xi_{1/2}f)^{\reg} + \lim_{T \to \infty}\int_{-1/2}^{1/2}\langle f(u+iT),\overline{\IKM(F,u + iT)}\rangle du.
		\end{align*}
		One can show as in \cite{alfes}, Theorem~5.1, that $\IKM(F,\tau)$ and $\IM(F,\tau)$ are orthogonal to cusp forms, i.e., $(\IKM(F,\tau),\xi_{1/2}f)^{\reg} = 0$. The integral on the right-hand side picks out the zero-coefficient, to which only the holomorphic part of $f$ contributes. Hence we obtain a formula for $(f,\theta_{1/2})^{\reg}$ of the shape \eqref{InnerProductEvaluation}, involving only the coefficients of $f^{+}$. Plugging in the coefficients of $\IKM(F,z)$ from Theorem~\ref{ThetaLifts} yields the result.
	\end{proof}

	\begin{Example}\label{ThetaNorms}
		As a simple application of the last result, we show that the Petersson norms of $\theta_{1/2}$ and $\theta_{3/2}$ are given by
		\begin{align*}
		(\theta_{1/2},\theta_{1/2}) = \frac{\pi(N+1)}{3\sqrt{N}} \qquad \text{and} \qquad ( \theta_{3/2},\theta_{3/2}) = \frac{\sqrt{N}(N-1)}{6}.
		\end{align*}
		 These can of course also be evaluated using more direct methods, for instance, the Rankin-Selberg $L$-function, but it is interesting to see how the dependency on $F$ in Theorem~\ref{InnerProductFormulas} disappears if we plug in $\theta_{1/2}$ or $\theta_{3/2}$ for $f$.
		
		We only show the formula for $\theta_{1/2}$, since the proof for $\theta_{3/2}$ is very similar. We can assume $a(0) = 1$. Let 
			\begin{align*}
		E_{2}^{*}(z) = 1-24\sum_{n=1}^{\infty}\sigma_{1}(n)e(nz)-\frac{\pi}{3y}, \qquad \bigg(\sigma_{1}(n) = \sum_{d \mid N}d\bigg),
		\end{align*}
		be the non-holomorphic Eisenstein series of weight $2$ for $\SL_{2}(\Z)$. Then $E_{2}^{*}(z)-NE_{2}^{*}(Nz)$ is a holomorphic modular form of weight $2$ for $\Gamma_{0}(N)$, and by applying the residue theorem to $F(z)(E_{2}^{*}(z)-NE_{2}^{*}(Nz))dz$, we find that $F$ satisfies
			\begin{align}\label{ResidueFormula}
			(1-N)-24\sum_{n>0}a(-n)(\sigma_{1}(n)-N\sigma_{1}(n/N)) = 0.
			\end{align}	
		 If we denote by $c_{\theta}(D,r)$ the coefficients of $\theta_{1/2}$, we see that $c_{\theta}(0,0) = 1$, and $c_{\theta}(b^{2},b) $ equals $2$ or $1$ for $b > 0$ depending on whether $b \equiv -b (2N)$ or not. Applying Theorem~\ref{InnerProductFormulas} and using the relation \eqref{ResidueFormula} we obtain
		\begin{align*}
		-\frac{\sqrt{N}}{4\pi}(\theta_{1/2},\theta_{1/2}) &= 4\sum_{n\geq 0}a(-n)\sigma_{1}(n)-4\sum_{\substack{b > 0 \\ b \equiv 0(N)}}b\sum_{n>0}a(-bn) - 2\sum_{\substack{b>0 \\ b \not \equiv 0 (N)}}b\sum_{n>0}a(-bn) \\
		&= 4\sigma_{1}(0) + 2\sum_{n>0}a(-n)(\sigma_{1}(n)-N\sigma_{1}(n/N)) = -\frac{1+N}{12}.
		\end{align*}
		This yields the stated formula.
	\end{Example}

		The formula given in the first item of Theorem~\ref{InnerProductFormulas} has applications in the theory of Borcherds products, see \cite{Borcherds}.  We follow the exposition of \cite{BruinierOnoHeegnerDivisors}. Let $f \in H_{1/2,\rho}$ be a harmonic Maass form of weight $1/2$ for $\rho$ whose shadow is a cusp form, and assume that $c_{f}^{+}(D,r) \in \R$ for all $D$ and $c_{f}^{+}(D,r) \in \Z$ for $D \leq 0$. Then the infinite product
		\begin{align*}
		\Psi(z,f) = e(\rho_{f,\infty}z)\prod_{n =1}^{\infty}(1-e(nz))^{c_{f}^{+}(n^{2},n)}
		\end{align*}
		is a meromorphic modular form of weight $c_{f}^{+}(0,0)$ for $\Gamma_{0}(N)$ and a unitary character, possibly of infinite order (see Theorems~6.1 and 6.2 in \cite{BruinierOnoHeegnerDivisors}). Here $\rho_{f,\infty}$ is the so-called \emph{Weyl vector} at $\infty$, which is defined by
		\[
		\rho_{f,\infty} = \frac{\sqrt{N}}{8\pi}(f,\theta_{1/2})^{\reg}.
		\]
		The Bocherds product $\Psi(z,f)$ has singularities at Heegner points in $\H$, which are prescribed by the principal part of $f$, and its orders at the cusps are determined by the corresponding Weyl vectors, which we describe now. 

		Each cusp of $\Gamma_{0}(N)$ can be represented by a reduced fraction $a/c$ with $c \mid N$, and the Weyl vector corresponding to $a/c$ is defined by
		\begin{align}\label{WeylVectorAtCusp}
		\rho_{f,a/c} = \frac{\sqrt{N}}{8\pi}(f,\theta_{1/2,N/(c,N/c)^{2}}^{\sigma_{c/(c,N/c)}}|U_{(c,N/c)})^{\reg},
		\end{align}
		where $\sigma_{c/(c,N/c)}$ is the Atkin-Lehner involution corresponding to the exact divisor $c/(c,N/c)$ of $N/(c,N/c)^{2}$ as in \eqref{AtkinLehnerInvolution}, and $U_{(c,N/c)}$ is the operator \eqref{UdOperator}. Note that the Weyl vector at $a/c$ does not depend on $a$. Further, Theorem~\ref{InnerProductFormulas} yields a formula for the Weyl vector at each cusp $a/c$, involving only the principal part of $f$ and the coefficients $c_{f}^{+}(b^{2},r)$ for $b > 0$ and $r \in \Z/2N\Z$ with $r^{2}\equiv b^{2}(4N)$. Thus, we obtain the following rationality result.
		
		\begin{Corollary}\label{WeylVectorsRationality}
			 Let $f \in H_{1/2,\rho}$ be a harmonic Maass form with $\xi_{1/2}f \in S_{3/2,\bar{\rho}}$. Suppose that $c_{f}^{+}(D,r) \in \R$ for all $D$ and that $c_{f}^{+}(D,r) \in \Z$ for $D \leq 0$. If $c_{f}^{+}(b^{2},r) \in \Q$ for all $b > 0$ and all possible $r \in \Z/2N\Z$, then the Weyl vectors $\rho_{f,a/c}$ at all cusps are rational.
		\end{Corollary}
		
		The formula for the Weyl vector $\rho_{f,a/c}$ obtained from Theorem \ref{InnerProductFormulas} looks quite complicated in general. Thus, for simplicity, we only state it in the special case of a cusp $a/c$ with $c \mid \mid N$ and $(a,c) = 1$. Then $\theta_{1/2,N/(c,N/c)^{2}}^{\sigma_{c/(c,N/c)}}|U_{(c,N/c)} = \theta_{1/2,N}^{\sigma_{c}}$, and Theorem \ref{InnerProductFormulas} gives the following formula.		
		
		\begin{Corollary}\label{WeylVectors}
			Let $f \in H_{1/2,\rho}$ be a harmonic Maass form with $\xi_{1/2}f \in S_{3/2,\bar{\rho}}$. Suppose that $c_{f}^{+}(D,r) \in \R$ for all $D$ and that $c_{f}^{+}(D,r) \in \Z$ for $D \leq 0$. Let $c\mid \mid N$ and let $\sigma_{c}$ be the associated Atkin-Lehner involution as in \eqref{AtkinLehnerInvolution}. Let $F \in M_{0}^{!,\infty}(N)$ be as in Lemma \ref{ConstructionModularFunction}, normalized to $a(0) = 1$. Then the Weyl vector $\rho_{f,a/c}$ at the cusp $a/c$ is given by
			\begin{align*}
			\rho_{f,a/c}= \frac{\sqrt{N}}{8\pi}(f^{\sigma_{c}},\theta_{1/2})^{\reg} &= -\frac{1}{2}\sum_{r(2N)}\sum_{\substack{D < 0\\ D \equiv r^{2}(4N)}}c_{f}^{+}(D,\sigma_{c}(r))\left( \tr_{F}^{+}(D,r) + \tr_{F}^{-}(D,r)\right) \\ 
			& \quad - 2c_{f}^{+}(0,0)\sum_{n\geq 0}a(-n)\sigma_{1}(n) + \sum_{b > 0}c_{f}^{+}(b^{2},\sigma_{c}(b))b\sum_{n > 0}a(-bn).
			\end{align*}
		\end{Corollary}
		
		\begin{Remark}
			\begin{enumerate}
				\item If $N$ is square free, the cusps of $\Gamma_{0}(N)$ are represented by the fractions $1/c$, where $c$ runs through the divisors of $N$. In this case all Weyl vectors can be computed with the above formula.
	\item In \cite{Borcherds}, Section 9, the Weyl vectors are computed in a similar way, using non-holomorphic Eisenstein series of weight $3/2$ as $\xi$-preimages for $\theta_{1/2}$. However, this only works if $N = 1$ or if $N = p$ is a prime. Otherwise, the Eisenstein series, and thus also its $\xi$-image, is invariant under all Atkin-Lehner involutions, but $\theta_{1/2}$ is not.
			\end{enumerate}
		\end{Remark}
		
		\begin{Example} We consider the Borcherds lift of $f = \theta_{1/2}$, for $N$ arbitrary. By Example~\ref{ThetaNorms}, the Weyl vector of $\theta_{1/2}$ at $\infty$ equals $(1+N)/24$, so its Borcherds product is given by $\Psi(z,\theta_{1/2}) = \eta(z)\eta(Nz)$. By a similar computation as in Example~\ref{ThetaNorms} we find 
		\[
		(\theta_{1/2}^{\sigma_{c}},\theta_{1/2}) =\frac{\pi}{3\sqrt{N}}\left(\frac{N}{c} + c\right)
		\]
		for $c \mid \mid N$. Hence the Weyl vector of $\Psi(z,f)$ at a cusp $a/c$ with $c \mid \mid N$ is given by $\frac{1}{24}\frac{N}{c}\big(1 + \frac{c^{2}}{N} \big)$.
		\end{Example}
		
		Finally, we would like to mention that the harmonic Maass form $\IM(F,\tau)$ given in Theorem~\ref{ThetaLifts} can be used to construct rational functions on $X_{0}(N)$ with special divisors. Let $\Delta \neq 1$ be a fundamental discriminant and let $r \in \Z$ with $\Delta \equiv r^{2}(4N)$. Further, let $\tilde{\rho} = \rho$ if $\Delta > 1$ and $\tilde{\rho} = \bar{\rho}$ if $\Delta < 0$. The twisted Borcherds product of a harmonic Maass form $f \in H_{1/2,\tilde{\rho}}$ with real holomorphic part and integral principal part is defined by
		\[
		\Psi_{\Delta,r}(z,f) = \prod_{n =1}^{\infty}\prod_{b(\Delta)}[1-e(b/\Delta)e(nz)]^{\left(\frac{\Delta}{b}\right)c_{f}^{+}(|\Delta|n^{2},rn)},
		\]
		see \cite{BruinierOnoHeegnerDivisors}, Theorem~6.1. Note that the Weyl vectors vanish for $\Delta \neq 1$. The function $\Psi_{\Delta,r}(z,f)$ is a meromorphic modular form of weight $0$ for $\Gamma_{0}(N)$ and a unitary character, which is of finite order if and only if the coefficients $c_{f}^{+}(|\Delta|n^{2},rn)$ are rational (see Theorem~6.2 in \cite{BruinierOnoHeegnerDivisors}). 
		
		If $F \in M_{0}^{!,\infty}(N)$ is as in Lemma~\ref{ConstructionModularFunction} and has integral principal part, then the Millson lift $\IM(F,\tau)$ given in Theorem~\ref{ThetaLifts} is a harmonic Maass form in $H_{1/2,\bar{\rho}}$ with rational holomorphic part and integral principal part. In particular, for $\Delta < 0$, some power of the twisted Borcherds product $\Psi_{\Delta,r}(z,\IM(F,\tau))$ defines a rational function on $X_{0}(N)$ whose zeros and poles lie on a twisted Heegner divisor.

\bibliography{references}{}
\bibliographystyle{plain}

\end{document}